\documentclass[letterpaper, 10 pt, conference]{ieeeconf}  % Comment this line out if you need a4paper

\IEEEoverridecommandlockouts                              % This command is only needed if 
\usepackage{graphicx} % for pdf, bitmapped graphics files
\usepackage{amsmath} % assumes amsmath package installed
\usepackage{amssymb}  % assumes amsmath package installed

\newtheorem{theorem}{Theorem}

\newtheorem{proposition}{Proposition}

\usepackage[noadjust]{cite}

\title{\LARGE \bf
Control Synthesis for Bilevel Linear Model Predictive Control%*
}

\author{Yonatan Mintz$^{1}$, John Audie Cabrera$^{2}$, Jhoanna Rhodette Pedrasa$^{2}$, and Anil Aswani$^{1}$% <-this % stops a space
\thanks{*This work was supported in part by the Philippine-California Advanced Research Institutes (PCARI) and NSF Award CMMI-1450963.}% <-this % stops a space
\thanks{$^{1}$Y. Mintz and A. Aswani are with the Department of Industrial Engineering and Operations Research, University of California, Berkeley, CA 94720 USA 
        {\tt\small ymintz@berkeley.edu, aaswani@berkeley.edu}}%
\thanks{$^{2}$J.A. Cabrera and J.R. Pedrasa are with the Electrical and Electronics Engineering Institute, University of the Philippines, Diliman, Quezon City, Philippines 1101
        {\tt\small john\_audie.cabrera@upd.edu.ph, jipedrasa@up.edu.ph}}%
}

\begin{document}

\maketitle
\thispagestyle{empty}
\pagestyle{empty}

%%%%%%%%%%%%%%%%%%%%%%%%%%%%%%%%%%%%%%%%%%%%%%%%%%%%%%%%%%%%%%%%%%%%%%%%%%%%%%%%
\begin{abstract}
Distributed model predictive control (MPC) is either cooperative or competitive, and control-theoretic properties have been less studied in the competitive (e.g., game theory) setting.  This paper studies MPC with linear dynamics and a Stackelberg game structure: Given a fixed lower-level linear MPC (LoMPC) controller, the bilevel linear MPC (BiMPC) controller chooses inputs to steer LoMPC knowing that LoMPC is optimizing with respect to a different cost function.  After defining LoMPC and BiMPC, we give examples to demonstrate how interconnections in a dynamic Stackelberg game can lead to loss/gain (as compared to the same system being centrally controlled) of controllability or stability.  Then, we give sufficient conditions under an arbitrary finite MPC horizon for stabilizability of BiMPC, and develop an approach to synthesize a stabilizing BiMPC controller.  Next, we define two (a duality-based technique and an integer-programming-based technique) reformulations to numerically solve the optimization problem associated with BiMPC, prove equivalence of these reformulations to BiMPC, and demonstrate their useful by simulations of a synthetic system and a case study of an electric utility changing electricity prices to perform demand response of a home's air conditioner controlled by a linear MPC.
\end{abstract}

%%%%%%%%%%%%%%%%%%%%%%%%%%%%%%%%%%%%%%%%%%%%%%%%%%%%%%%%%%%%%%%%%%%%%%%%%%%%%%%%

\section{Introduction}

Distributed model predictive control (MPC) is classified by information flows and the order of controller computations. Most work \cite{camponogara2002,venkat2005,rawlings2009,scattolini2009,raimondo2009,ma2011,farina2011,ferramosca2013} studies how to decentralize solution of the MPC optimization problem. Hierarchical MPC \cite{scattolini2009,scattolini2007,picasso2010,vermillion2014} has interactions between a supervisory MPC layer and a low-level MPC layer, and both layers are engineered to ensure closed-loop stability.  In contrast, noncooperative MPC \cite{venkat2005,rawlings2009} has several coequal MPC controllers that have differing objective functions. Thus, noncooperative MPC has a game-theoretic interpretation: The stationary solution of the controllers is a Nash equilibrium, which means noncooperative MPC can model competition between agents/systems.

Stackelberg games \cite{von1952} have a leader-follower interconnection structure that has \emph{not} been well-studied in the context of distributed MPC.  In these games, the follower's controller is fixed and the leader engineers their own controller.  It differs from hierarchical MPC \cite{scattolini2009,scattolini2007,picasso2010,vermillion2014} in that only the leader's controller is engineered and the follower's controller may be unstable, and it differs from noncooperative MPC \cite{venkat2005,rawlings2009} in that the follower gets to first observe the leader's control actions and then choose their own control.%The differences from hierarchical MPC are that in Stackelberg games only the leader's controller is engineered and the follower's controller may be unstable.  The difference from noncooperative MPC is that in Stackelberg games the follower gets to first observe the leader's control actions and then choose their own control.

Though Stackelberg games have been used in control applications featuring human-automation interactions \cite{basar1979,li2002,aswani2011,zhu2011,vasudevan2012safe,krichene2014,jamaludin2015bilevel,sadigh2016planning}, little attention has been paid towards controllability, stability, and controller synthesis.  Stackelberg games are bilevel programs \cite{colson2007,aswani2015,aswani2016}, which are optimization problems where some constraints are the solutions to a \emph{lower-level} optimization problem.  This perspective of bilevel programs provides a promising framework from which to study control-theoretic questions (like stability and synthesis) of Stackelberg games.

%This paper defines and studies bilevel MPC, which is a distributed MPC with a Stackelberg game structure.  We first provide a formulation of bilevel MPC, and then give counterexamples that demonstrate how the interconnection structure in bilevel MPC can lead to complex behaviors like loss/gain of controllability or loss/gain of stability.  The next section provides sufficient conditions under specific dynamics and cost functions for stabilizability of bilevel MPC, stability of a given bilevel MPC controller, and an approach for synthesis of a stable bilevel MPC controller.  We conclude with one synthetic example to demonstrate a duality-based approach \cite{aswani2016} to solving the optimization problem for bilevel MPC, and one simulation example to demonstrate an integer programming approach to solving the optimization problem for the bilevel MPC of an electric utility changing electricity prices to minimize peak demand of a home air conditioner controlled by linear MPC.

This paper defines and then studies control-theoretic properties of bilevel linear MPC (BiMPC), which is a distributed linear MPC with Stackelberg game structure.  The idea is that the lower-level linear MPC is a model of either the human decision-making process \cite{aswani2011,vasudevan2012safe,krichene2014,sadigh2016planning,aswani2016behavioral} or of an automated system \cite{ma2011,deng2010building,aswani2012,aswani2012energy,aswani2012identifying,jamaludin2015bilevel,he2016}.  And our goal in designing BiMPC is to engineer the system in order to steer the lower-level MPC towards desired configurations.

We first define BiMPC, and then give examples that show how interconnections in dynamic Stackelberg games can lead to loss/gain (as compared to the same dynamics being centrally controlled) of controllability or stability.  Next, we provide sufficient conditions for stabilizability of BiMPC.  An approach to synthesize a stable BiMPC controller is also derived.  We then define two reformulations (based on duality theory \cite{aswani2015,aswani2016} and integer-programming \cite{aswani2016behavioral,aswani2016b}) to numerically solve the optimization problem associated with BiMPC, prove equivalence of these reformulations to BiMPC, and demonstrate their usefulness by simulations of a synthetic system and a case study of an electric utility changing electricity prices to perform demand response of a home's air conditioner controlled by a linear MPC.

\section{Formulation of Bilevel Linear MPC}
Let $\xi\in\mathbb{R}^p$ and $\nu\in\mathbb{R}^q$, and suppose the overall control system is linear $\xi_+ = A\xi+B\nu$ with matrices of dimensions $A\in\mathbb{R}^{p\times p}$ and $B\in\mathbb{R}^{p\times q}$.  We decompose the state space as $\xi^\textsf{T} = \begin{bmatrix}x^\textsf{T} & y^\textsf{T}\end{bmatrix}$ with $x\in\mathbb{R}^{\rho}$ and $y\in\mathbb{R}^{p-\rho}$, and we decompose the input as $\nu^\textsf{T} = \begin{bmatrix}u^\textsf{T} & w^\textsf{T}\end{bmatrix}$ with $u\in\mathbb{R}^{\gamma}$ and $w\in\mathbb{R}^{q-\gamma}$.  It is also useful to decompose $B$ as $B = \begin{bmatrix}B_1 & B_2\end{bmatrix}$ with $B_1 \in\mathbb{R}^{p\times\gamma}$ and $B_2\in\mathbb{R}^{p\times(q-\gamma)}$.  Lastly, the sets $\mathcal{X},\mathcal{Y},\mathcal{U},\mathcal{W}$ are compact, non-singleton polytopes that contain the origin.  These sets are assumed to be characterized by a finite number of linear inequalities, and they are used to provide constraints on $x,y,u,w$, respectively. 
% the dimensions of $A,B_1,B_2$ are consistent with $\xi,u,v$.

Now let $\langle r\rangle = \{0,\ldots,r-1\}$ and $[r] = \{1,\ldots,r\}$, define the positive semidefinite matrices $U,V\in\mathbb{R}^{p\times p}$, and define the matrix $W\in\mathbb{R}^{q\times q}$ that decomposes as
\begin{equation}
W = \begin{bmatrix}W_1 & \Phi\\\Phi^\textsf{T} & W_2\end{bmatrix}
\end{equation}
with the matrix $W_2\in\mathbb{R}^{(p-\rho)\times(p-\rho)}$ assumed to be positive definite.  To simplify our notation for defining MPC with a time horizon of $N$ time steps, we use $\boldsymbol\xi = \{\xi_1,\ldots,\xi_N\}$, $\mathbf{u} = \{u_0,\ldots,u_{N-1}\}$, and $\mathbf{w} = \{w_0,\ldots,w_{N-1}\}$.  We define the lower-level linear model predictive control (LoMPC) problem with a horizon of $N$ to be
\begin{equation}
\begin{aligned}
\mathbf{P}_L(\xi_0,\mathbf{u}) = \min_{\boldsymbol\xi,\mathbf{w}}\ & \xi_N^\textsf{T}U\xi_N^{\vphantom{\textsf{T}}} + \textstyle\sum_{n=0}^{N-1} \xi_n^\textsf{T}V\xi_n^{\vphantom{\textsf{T}}} + \nu_n^\textsf{T}W\nu_n^{\vphantom{\textsf{T}}}\\
\text{s.t. }& \xi_{n+1} = A\xi_n + B\nu_n \text{ for } n \in\langle N\rangle\\
&y_n \in\mathcal{Y},\, w_n\in\mathcal{W} \hspace{0.145cm} \text{ for } n\in\langle N\rangle\\
&y_N\in\mathcal{Y}_\Omega
\end{aligned}
\end{equation}
where $\mathcal{Y}_\Omega$ is a positively robust invariant set such that given a matrix $K_L\in\mathbb{R}^{(q-\gamma)\times p}$ then the set $\mathcal{Y}_\Omega$ satisfies: (i) $\mathcal{Y}_\Omega\subseteq\mathcal{Y}$ and $K_L\xi\in\mathcal{W}$ for $(x,y)\in\mathcal{X}\times\mathcal{Y}_\Omega$, and (ii) $(A+B_2K_L)\xi + B_1u \in \mathcal{Y}_\Omega$ for $(x,y,u)\in\mathcal{X}\times\mathcal{Y}_\Omega\times\mathcal{U}$.  The set $\mathcal{Y}_\Omega$ can be computed by existing algorithms \cite{mayne1997,kolmanovsky1998,blanchini1999,rakovic2005,mohan2016convex}. 
%\begin{equation}
%\label{eqn:yo}
%\begin{aligned}
%&\mathcal{Y}_\Omega\subseteq\mathcal{Y} \text{ and } \ell_L(y)\in\mathcal{V} \text{ for } y\in\mathcal{Y}_\Omega\\
%&A\xi + B_1u + B_2\ell_L(y) \in \mathcal{Y}_\Omega \text{ for } (x,y,u)\in\overline{\mathcal{X}}\times\mathcal{Y}_\Omega\times\overline{\mathcal{U}}
%\end{aligned}
%\end{equation}
%The set $\mathcal{Y}_\Omega$ can be computed using existing algorithms \cite{mayne1997,kolmanovsky1998,blanchini1999,rakovic2005}. 

Next let $P,Q\in\mathbb{R}^{p\times p}$, $R\in\mathbb{R}^{q\times q}$ be positive semidefinite matrices.  
The bilevel linear model predictive control (BiMPC) problem with a horizon of $N$ is given by
\begin{equation}
\begin{aligned}
\mathbf{P}_B(\xi_0) = \min_{\boldsymbol\xi,\mathbf{u},\mathbf{w}}\ & \xi_N^\textsf{T}P\xi_N^{\vphantom{\textsf{T}}} + \textstyle\sum_{n=0}^{N-1} \xi_n^\textsf{T}Q\xi_n^{\vphantom{\textsf{T}}} + \nu_n^\textsf{T}R\nu_n^{\vphantom{\textsf{T}}}\\
\text{s.t. }& (\boldsymbol\xi,\mathbf{w})\in\arg\min_{\boldsymbol\xi,\mathbf{w}} \mathbf{P}_L(\xi_0,\mathbf{u})\\
%&\xi_{1}^\textsf{T}F\xi_1^{\vphantom{\textsf{T}}} \leq \xi_{0}^\textsf{T}(F-H)\xi_0^{\vphantom{\textsf{T}}}\\
&h(\xi_0,\boldsymbol\xi) \leq 0\\
&x_n \in\mathcal{X},\, u_n\in\mathcal{U} \text{ for } n\in\langle N\rangle\\
&x_N\in\mathcal{X}
\end{aligned}
\end{equation}
where $h(\xi_0,\boldsymbol\xi)\leq 0$ is a constraint that will be designed in Sect. \ref{section:scfs} to ensure recursive feasibility and stability.  This constraint is similar to a Lyapunov constraint that has been used in certain MPC schemes to ensure stability \cite{lu2000quasi,mhaskar2006stabilization}, though a conceptual difference is that our constraint is needed for both recursive feasibility and stability.

\section{Interconnection Examples}

%The interconnection between the upper and lower dynamics in BiMPC leads to complex behaviors not typically encountered in hierarchical or noncooperative control.  Our first examples show that interconnection can cause a loss or gain of controllability, and our next examples show that interconnection can lead to a gain or loss of stability.

Interconnections in dynamic Stackelberg games can cause a loss/gain of controllability or stability as compared to the same dynamics when they are centrally controlled by $u$ with $w\equiv 0$, which is behavior not often seen in hierarchical or noncooperative control.  Our examples use $h(\xi_0,\boldsymbol\xi)\equiv 0$ and $\mathcal{X}=\mathcal{Y}=\mathcal{Y}_\Omega=\mathcal{U}=\mathcal{W} = \mathbb{R}$, and we refer to the dynamics on the $x$ states ($y$ states) as the upper (lower) dynamics.

%Interconnection in BiMPC between the upper and lower dynamics leads to behavior often not seen in hierarchical or noncooperative control.  We show the interconnection in BiMPC can cause a loss/gain of controllability or stability.

\subsection{Controllability Examples}

Our first example is a problem where the LoMPC is
\begin{equation}
\begin{aligned}
\mathbf{P}_L(\xi_0,\mathbf{u}) = \min_{\boldsymbol\xi,\mathbf{w}}\ & (u_0 + w_0)^2\\
\text{s.t. }& x_1 = 2x_0 + y_0\\
&y_1 = y_0 + u_0+w_0 \\
&y_1 \in\mathbb{R}, w_0\in\mathbb{R}
\end{aligned}
\end{equation}
The overall system is controllable in $u$ when $w\equiv 0$.  But a simple calculation shows the control of LoMPC is $w = -u$, and so for $n\geq 1$ the dynamics seen by BiMPC are
\begin{equation}
\begin{aligned}
&x_{n+1} = 2x_n +y_n\\
&y_{n+1} = y_n
\end{aligned}
\end{equation}
which is not controllable in $u$.  The control action of LoMPC leads to a loss of controllability by BiMPC in this example.

The next example is a problem where the LoMPC is
\begin{equation}
\begin{aligned}
\mathbf{P}_L(\xi_0,\mathbf{u}) = \min_{\boldsymbol\xi,\mathbf{w}}\ & (y_1)^2 + (w_0)^2\\
\text{s.t. }& x_1 = 2x_0 + w_0\\
&y_1 = y_0 + u_0+w_0 \\
&y_1 \in\mathbb{R}, w_0\in\mathbb{R}
\end{aligned}
\end{equation}
The overall system is not controllable in $u$ when $w\equiv 0$.  But a simple calculation shows the control of LoMPC is $w = -(y+u)/2$, and so the dynamics seen by BiMPC are
\begin{equation}
\begin{bmatrix}x_{n+1}\\y_{n+1}\end{bmatrix} = \frac{1}{2}\cdot\begin{bmatrix}4 & -1\\0 & \hphantom{-}1\end{bmatrix}\begin{bmatrix}x_n\\y_n\end{bmatrix}+\frac{1}{2}\cdot\begin{bmatrix}-1\\\hphantom{-}1\end{bmatrix}u_n
\end{equation}
which is controllable in $u$.  The control action of LoMPC leads to a gain of controllability by BiMPC in this example.%This example is interesting because the control action of LoMPC leads to a gain of controllability by BiMPC.

\subsection{Stability Examples}

Our first example is a problem where the LoMPC is
\begin{equation}
\begin{aligned}
\mathbf{P}_L(\xi_0,\mathbf{u}) = \min_{\boldsymbol\xi,\mathbf{w}}\ & (y_1)^2 + (w_0)^2\\
\text{s.t. }& x_1 = x_0 + 2y_0 + u_0\\
&y_1 = 2x_0 + y_0 + u_0+w_0 \\
&y_1 \in\mathbb{R}, w_0\in\mathbb{R}
\end{aligned}
\end{equation}
and the BiMPC is
\begin{equation}
\begin{aligned}
\mathbf{P}_B(\xi_0) = \min_{\boldsymbol\xi,\mathbf{u},\mathbf{w}}\ & (x_1)^2 + (u_0)^2\\
\text{s.t. } & x_1 \in\mathbb{R}, u_0\in\mathbb{R}\\
& (\xi_1,w_0) \in \arg\min_{\boldsymbol\xi,\mathbf{w}}\mathbf{P}_L(\xi_0,\mathbf{u})\\
\end{aligned}
\end{equation}
A simple calculation gives that the closed loop system is
\begin{equation}
\begin{bmatrix}x_{n+1}\\y_{n+1}\end{bmatrix} = \frac{1}{4}\cdot\begin{bmatrix}2 & 4\\3&0\end{bmatrix}\begin{bmatrix}x_n\\y_n\end{bmatrix}
\end{equation}
which is unstable.  The lower dynamics are stable when $(x,u) = (0,0)$, and the upper dynamics are stable when $(y,w) = (0,0)$; yet the overall control system is unstable.  But by changing the objective function of the BiMPC to $(x_1)^2 + (u_0)^2 + (y_1)^2$, the closed loop system becomes
\begin{equation}
\begin{bmatrix}x_{n+1}\\y_{n+1}\end{bmatrix} = \frac{1}{5}\cdot\begin{bmatrix}1 & 5\\3&0\end{bmatrix}\begin{bmatrix}x_n\\y_n\end{bmatrix}
\end{equation}
which is stable.  Thus stability of the overall control system is dependent on the gains of the upper and lower dynamics.

As another example, consider the LoMPC given by
\begin{equation}
\begin{aligned}
\mathbf{P}_L(\xi_0,\mathbf{u}) = \min_{\boldsymbol\xi,\mathbf{w}}\ & (w_0)^2\\
\text{s.t. }& x_1 = x_0 + 4y_0 + u_0\\
&y_1 = 4y_0 + u_0+w_0 \\
&y_1 \in\mathbb{R}, w_0\in\mathbb{R}
\end{aligned}
\end{equation}
and the BiMPC is
\begin{equation}
\begin{aligned}
\mathbf{P}_B(\xi_0) = \min_{\boldsymbol\xi,\mathbf{u},\mathbf{w}}\ & (x_1)^2 + (u_0)^2\\
\text{s.t. } &x_1 \in\mathbb{R}, u_0\in\mathbb{R}\\
& (\xi_1,w_0) \in \arg\min_{\boldsymbol\xi,\mathbf{w}}\mathbf{P}_L(\xi_0,\mathbf{u})\\
\end{aligned}
\end{equation}
A simple calculation gives that the closed loop system is
\begin{equation}
\begin{bmatrix}x_{n+1}\\y_{n+1}\end{bmatrix} = \frac{1}{2}\cdot\begin{bmatrix}\hphantom{-}1 & 4\\-1&4\end{bmatrix}\begin{bmatrix}x_n\\y_n\end{bmatrix}
\end{equation}
which is unstable. This example is interesting because the control provided by LoMPC (the control is always $w\equiv 0$) is never stabilizing, while the control of BiMPC stabilizes the upper dynamics when $(y,w) = (0,0)$.  On the other hand, when the objective function of BiMPC is changed to $(x_1)^2 + (u_0)^2 + 3(y_1)^2$, the closed loop system is
\begin{equation}
\begin{bmatrix}x_{n+1}\\y_{n+1}\end{bmatrix} = \frac{1}{5}\cdot\begin{bmatrix}\hphantom{-}4 & 4\\-1&4\end{bmatrix}\begin{bmatrix}x_n\\y_n\end{bmatrix}
\end{equation}
which is stable.  This example shows that in certain situations the BiMPC can stabilize the overall control system independent of the control action provided by LoMPC.

\section{Sufficient Condition For Stability}
\label{section:scfs}

The above examples show that stability of BiMPC depends non-trivially on the dynamics and cost functions, and so we focus on providing sufficient conditions for stabilizability and then develop an approach for controller synthesis in BiMPC.

It is helpful to define some additional notation.  
Let $\Lambda_N = U$, define the matrices
\begin{equation}
\begin{aligned}
&\Theta_n{\vphantom{\textsf{T}}} = W_2^{\vphantom{\textsf{T}}}+B_2^\textsf{T}\Lambda_n^{\vphantom{\textsf{T}}}B_2^{\vphantom{\textsf{T}}}\\
&\Lambda_{n-1}{\vphantom{\textsf{T}}} = A^\textsf{T}\Lambda_n^{\vphantom{\textsf{T}}}A - A^\textsf{T}\Lambda_n^{\vphantom{\textsf{T}}}B_2^{\vphantom{\textsf{T}}}\Theta_n^{-1}B_2^\textsf{T}\Lambda_n^{\vphantom{\textsf{T}}}A + V\\
&\Psi_{n-1}{\vphantom{\textsf{T}}} = -\Theta_n^{-1}B_2^\textsf{T}\Lambda_n^{\vphantom{\textsf{T}}}A_{\vphantom{1}}^{\vphantom{\textsf{T}}}
\end{aligned}
\end{equation}
for $n\in[N]$, and define the matrix $\Gamma^{\vphantom{\textsf{T}}} = \mathbb{I} - B_2^{\vphantom{\textsf{T}}}\Theta_1^{-1}B_2^\textsf{T}\Lambda_1^{\vphantom{\textsf{T}}}$.
%\begin{align}
%&\Psi = (W+B_2^\textsf{T}\Lambda_1^{\vphantom{\textsf{T}}}B_2^{\vphantom{\textsf{T}}})^{-1}B_2^\textsf{T}\Lambda_1^{\vphantom{\textsf{T}}}\\
%%&\Gamma = \mathbb{I} - B_2(W+B_2^\textsf{T}\Lambda_1^{\vphantom{\textsf{T}}}B_2^{\vphantom{\textsf{T}}})^{-1}B_2^\textsf{T}\Lambda_1^{\vphantom{\textsf{T}}}.
%&\Gamma = \mathbb{I} - B_2\Psi.
%\end{align}
Our first result concerns the properties of a specific set.

\begin{proposition}
Let $G\in\mathbb{R}^{\gamma\times p}$ be any matrix.  Then the set 
\begin{equation}
\begin{aligned}
\Xi = \Big\{\xi_0 :\ &G\xi_0 \in \mathcal{U}  \\
&\Theta_1^{-1}(B_2^\textsf{T}\Lambda_1(A+B_1^{\vphantom{\textsf{T}}}G) + \Phi^\textsf{T} G)\xi_0\in\mathcal{W}\\
&\Psi_n\xi_n\in\mathcal{W} \text{ for } n\in[N-1] \\
&\xi_1 = (\Gamma A+(\Gamma B_1-B_2^{\vphantom{\textsf{T}}}\Theta_1^{-1}\Phi^\textsf{T})G)\xi_0^{\vphantom{\textsf{T}}}\\
&\xi_{n+1} = (A+ B_2\Psi_n)\xi_n \text{ for } n\in[N-1]\\
&x_n\in\mathcal{X}, \, y_n \in\mathcal{Y}, \text{ for } n\in\langle N\rangle\\
&x_N\in\mathcal{X}, \, y_N\in\mathcal{Y}_\Omega\Big\}
\end{aligned}
\end{equation}
is non-singleton and contains the origin if $\mathcal{Y}_\Omega$ is non-singleton and contains the origin.
\end{proposition}

\begin{proof}
We start by proving the origin belongs to $\Xi$.  Note that if $\xi_0 = 0$, then $G\xi_0 = 0$, $\Theta_1^{-1}(B_2^\textsf{T}\Lambda_1(A+B_1^{\vphantom{\textsf{T}}}G) + \Phi^\textsf{T} G)\xi_0 = 0$, $\xi_1 = (\Gamma A+(\Gamma B_1-B_2^{\vphantom{\textsf{T}}}\Theta_1^{-1}\Phi^\textsf{T})G)\xi_0^{\vphantom{\textsf{T}}} = 0$, $\xi_{n+1} = (A+ B_2\Psi_n)\xi_n = 0$ for $n\in[N-1]$, and $\Psi_n\xi_n = 0$ for $n\in[N-1]$.  Thus $0\in\Xi$ since the sets $\mathcal{X},\mathcal{Y},\mathcal{U},\mathcal{W}$ contain the origin.  Next we prove $\Xi$ is non-singleton.  Since $N$ is finite, this means $\Lambda_n,\Psi_n$ for $n\in\langle N\rangle$ and $\Psi_N$ are finite.  So we can pick an $r > 0$ such that $\{\xi_0 : \xi_0^\textsf{T}\xi_0^{\vphantom{\textsf{T}}} \leq r\} \subseteq\Xi$ since $G$, $\Theta_1^{-1}(B_2^\textsf{T}\Lambda_1(A+B_1^{\vphantom{\textsf{T}}}G) + \Phi^\textsf{T} G)$, $\Psi_n$ for $n\in[N-1]$, $\Gamma A+(\Gamma B_1-B_2^{\vphantom{\textsf{T}}}\Theta_1^{-1}\Phi^\textsf{T})G$, and $A+B_2\Psi_n$ for $n\in[N-1]$ have finite norm.  So $\Xi$ is non-singleton.
\end{proof}

With the above definied matrices and set, we can now study stabilizability and controller synthesis for BiMPC.

\begin{theorem}
\label{thm:synstab}
If $(\Gamma A, \Gamma B_1-B_2^{\vphantom{\textsf{T}}}\Theta_1^{-1}\Phi^\textsf{T})$ is stabilizable, then BiMPC is stabilizable.  In particular, let $G\in\mathbb{R}^{\gamma\times p}$ be any matrix so $Z = \Gamma A+(\Gamma B_1-B_2^{\vphantom{\textsf{T}}}\Theta_1^{-1}\Phi^\textsf{T})G$ is Schur stable, and let $H\in\mathbb{R}^{p\times p}$ be the unique positive definite matrix that solves the discrete time Lyapunov equation $Z^\textsf{T}HZ - H = -\mathbb{I}$.  If $\xi_0\in\Xi$ and $X = \{\xi : \xi^\textsf{T}H\xi \leq \xi_0^\textsf{T}H\xi_0^{\vphantom{\textsf{T}}}\} \subseteq \Xi$, then we have that BiMPC with the choice
%For any $r > 0$ such that $X(r) = \{\xi : \xi^\textsf{T}\xi \leq r\} \subseteq \Xi$, we have that for $\xi_0 \in\Xi\cap X(r)$ the BiMPC with the choice
\begin{equation}
h(\xi_0,\boldsymbol\xi) = \xi_1^\textsf{T}H\xi_1^{\vphantom{\textsf{T}}} - \xi_0^\textsf{T}(H-\mathbb{I})\xi_0^{\vphantom{\textsf{T}}}
\end{equation}
stabilizes the control system, is recursively feasible, and ensures $\xi_n \in \mathcal{X}\times\mathcal{Y}$, $u_n\in\mathcal{U}$, and $w_n\in\mathcal{W}$ for all $n\geq 0$.
\end{theorem}

\begin{proof}
Suppose $\mathbf{u} = \{G\xi_0,0,\ldots,0\}$.  Then a dynamic programming calculation on LoMPC gives $w_n = \Psi_nA\xi_n$ and $\xi_{n+1} = (A+B_2\Psi_nA)\xi_n$ for $n\in[N-1]$, and that
\begin{equation}
\label{eqn:dponestep}
\begin{aligned}
\mathbf{P}_L(\xi_0,\mathbf{u}) = \min_{\xi_1,w_0}\ & \xi_1^\textsf{T}\Lambda_1^{\vphantom{\textsf{T}}}\xi_1^{\vphantom{\textsf{T}}} + \xi_0^\textsf{T}V\xi_0^{\vphantom{\textsf{T}}} + \nu_0^\textsf{T}W\nu_0^{\vphantom{\textsf{T}}}\\
\text{s.t. }& \xi_1 = A\xi_0 + B\nu_0\\
&y_0,y_1 \in\mathcal{Y},\, w_0\in\mathcal{W}
\end{aligned}
\end{equation}
when $w_{n+1}\in\mathcal{W}$ for $n\in\langle N-1\rangle$, $y_{n+2}\in\mathcal{Y}$ for $n\in\langle N-2\rangle$, and $y_N\in\mathcal{Y}_\Omega$.  Since $u_0 = G\xi_0$, solving (\ref{eqn:dponestep}) gives 
\begin{equation}
\begin{aligned}
&w_0 = \Theta_1^{-1}(B_2^\textsf{T}\Lambda_1(A+B_1^{\vphantom{\textsf{T}}}G) + \Phi^\textsf{T} G)\xi_0\\
&\xi_1 = (\Gamma A+(\Gamma B_1-B_2^{\vphantom{\textsf{T}}}\Theta_1^{-1}\Phi^\textsf{T})G)\xi_0
\end{aligned}
\end{equation}
when $w_0\in\mathcal{W}$ and $y_0,y_1\in\mathcal{Y}$.  But if $\xi_0\in\Xi$, then the above described requirements on $y_N$ and $w_n,y_n$ for $n\in\langle N-1\rangle$ hold by definition of $\Xi$.  And so the above values of $\boldsymbol\xi, \mathbf{w}$ are in fact the minimizers of LoMPC when $\mathbf{u} = \{G\xi_0,0,\ldots,0\}$ and $\xi_0\in\Xi$, which implies the above $\boldsymbol\xi, \mathbf{u}, \mathbf{w}$ are feasible for BiMPC since $h(\xi_0,\boldsymbol\xi) = \xi_0^\textsf{T}(\Gamma A+(\Gamma B_1-B_2^{\vphantom{\textsf{T}}}\Theta_1^{-1}\Phi^\textsf{T})G)^TH(\Gamma A+(\Gamma B_1-B_2^{\vphantom{\textsf{T}}}\Theta_1^{-1}\Phi^\textsf{T})G)\xi_0 - \xi_0^\textsf{T}(H-\mathbb{I})\xi_0^{\vphantom{\textsf{T}}}=0$ by the discrete time Lyapunov equation.

%Now consider the (possibly different) values $\boldsymbol\xi, \mathbf{u}, \mathbf{w}$ that are optimal for BiMPC.  This minimizer exists because we showed above that BiMPC was feasible when $\xi_0 \in\Xi$.  By definition of LoMPC and BiMPC we have $u_0\in\mathcal{U}$, $w_0\in\mathcal{W}$, and $\xi_1\in\mathcal{X}\times\mathcal{Y}$ when we use the optimal $\boldsymbol\xi, \mathbf{u}, \mathbf{w}$.  Furthermore, our choice of $h(\xi_0,\boldsymbol\xi)$ gives that $\xi_1^\textsf{T}H\xi_1^{\vphantom{\textsf{T}}} \leq \xi_0^\textsf{T}H\xi_0^{\vphantom{\textsf{T}}} - \xi_0^\textsf{T}\xi_0^{\vphantom{\textsf{T}}} \leq \xi_0^\textsf{T}H\xi_0^{\vphantom{\textsf{T}}}$.  This means $\xi_1^{\vphantom{\textsf{T}}}\in X$, and so by assumption we have $\xi_1\in\Xi$ since we assumed $X\subseteq\Xi$.  Using the same argument as above, this implies BiMPC is feasible for $\xi_1$.  This proves recursive feasibility and recursive constraint satisfaction.  Stability of BiMPC follows by noting that $\xi^\textsf{T}H\xi$ is Lyapunov function for the control action provided by BiMPC.
Now consider the (possibly different) values $\boldsymbol\xi, \mathbf{u}, \mathbf{w}$ that are optimal for BiMPC.  This minimizer exists because we showed that BiMPC is feasible when $\xi_0 \in\Xi$.  By definition of LoMPC and BiMPC we have $u_0\in\mathcal{U}$, $w_0\in\mathcal{W}$, and $\xi_1\in\mathcal{X}\times\mathcal{Y}$ when we use the optimal $\boldsymbol\xi, \mathbf{u}, \mathbf{w}$.  Furthermore, our choice of $h(\xi_0,\boldsymbol\xi)$ gives that $\xi_1^\textsf{T}H\xi_1^{\vphantom{\textsf{T}}} \leq \xi_0^\textsf{T}H\xi_0^{\vphantom{\textsf{T}}} - \xi_0^\textsf{T}\xi_0^{\vphantom{\textsf{T}}} \leq \xi_0^\textsf{T}H\xi_0^{\vphantom{\textsf{T}}}$.  This means $\xi_1^{\vphantom{\textsf{T}}}\in X$, and so by assumption we have $\xi_1\in\Xi$ since we assumed $X\subseteq\Xi$.  Using the same argument as above, this implies BiMPC is feasible for $\xi_1$.  This proves recursive feasibility and recursive constraint satisfaction.  Stability of BiMPC follows by noting $\xi^\textsf{T}H\xi$ is a Lyapunov function for the control provided by BiMPC.
\end{proof}

Observe that this result gives a method for synthesizing a controller because $\Gamma,\Theta_1$ are both constant matrices that can be computed using matrix operations on $A,B,U,V,W$, and so controller design for BiMPC consists of appropriately choosing the matrices $G, P,Q,R$ and computing the matrix $H$ by solving a discrete time Lyapunov equation.

\section{Duality Approach to Solving BiMPC}

New algorithms that use duality theory to solve bilevel programs have recently been proposed \cite{aswani2015,aswani2016}, and here we describe how to adapt these approaches to solve BiMPC. 

\subsection{Duality-Based Reformulation of BiMPC}

We first specify some notation: Define $\|\xi\|_M^2 = \xi^\textsf{T}M\xi$ for a matrix $M$, $\mathcal{X} = \{\xi: F_x\xi \leq g_x\}$, $\mathcal{Y} = \{\xi : F_y\xi \leq g_y\}$, $\mathcal{Y}_\Omega = \{\xi : F_o\xi \leq g_o\}$, $\mathcal{U} = \{\nu : F_u\nu \leq g_u\}$, $\mathcal{W} = \{w : F_ww\leq g_w\}$, $\boldsymbol\mu = \{\mu_0,\ldots,\mu_{N-1}\}$, $\boldsymbol\lambda = \{\lambda_0,\ldots,\lambda_{N}\}$, and $\boldsymbol\gamma = \{\gamma_0,\ldots,\gamma_{N-1}\}$.  With this notation, we next present our duality-based reformulation of BiMPC:
\begin{multline}
\label{eqn:dbreform}
\mathbf{P}_{DB}(\xi_0, \epsilon) = \\
\begin{aligned}
\min_{\substack{\boldsymbol\xi,\mathbf{u},\mathbf{w}\\\boldsymbol\mu,\boldsymbol\lambda,\boldsymbol\gamma}}\ & \xi_N^\textsf{T}P\xi_N^{\vphantom{\textsf{T}}} + \textstyle\sum_{n=0}^{N-1} \xi_n^\textsf{T}Q\xi_n^{\vphantom{\textsf{T}}} + \nu_n^\textsf{T}R\nu_n^{\vphantom{\textsf{T}}}\\
\text{s.t. }& \xi_N^\textsf{T}U\xi_N^{\vphantom{\textsf{T}}} + \textstyle\sum_{n=0}^{N-1}\xi_n^\textsf{T}V\xi_n^{\vphantom{\textsf{T}}} + \nu_n^\textsf{T}W\nu_n^{\vphantom{\textsf{T}}}+\\
&\hspace{2.5cm} - \zeta(\xi_0,\mathbf{u},\boldsymbol\mu,\boldsymbol\lambda,\boldsymbol\gamma) \leq \epsilon\\
&\xi_{n+1} = A\xi_n + B\nu_n \text{ for } n \in\langle N\rangle\\
&\xi_1^\textsf{T}H\xi_1^{\vphantom{\textsf{T}}} - \xi_0^\textsf{T}(H-\mathbb{I})\xi_0^{\vphantom{\textsf{T}}}\leq 0\\
&F_xx_n\leq g_x,\, F_yy_n\leq g_y \text{ for } n\in\langle N\rangle\\
&F_uu_n\leq h_u,\, F_ww_n\leq h_w \text{ for } n\in\langle N\rangle\\
&F_xx_N\leq g_x,\, F_oy_N\leq g_o\\
&\lambda_n \geq 0,\, \gamma_n \geq 0 \text{ for } n\in\langle N\rangle
\end{aligned}
\end{multline}
where $\lambda_n, \mu_n, \gamma_n$ have appropriate dimensions to define
\begin{multline}
\label{eqn:langdualfun}
\zeta(\xi_0,\mathbf{u},\boldsymbol\mu,\boldsymbol\lambda,\boldsymbol\gamma) = \textstyle-\frac{1}{4}\|F_o^\textsf{T}\lambda_N^{\vphantom{\textsf{T}}} + \mu_{N-1}^{\vphantom{\textsf{T}}}\|_{U^\dag}^2 - g_o^\textsf{T}\lambda_N^{\vphantom{\textsf{T}}}+\\
\textstyle\sum_{n=1}^{N-1} \Big[-\frac{1}{4}\|F_y^\textsf{T}\lambda_n^{\vphantom{\textsf{T}}} + \mu_{n-1}^{\vphantom{\textsf{T}}} - A^\textsf{T}\mu_{n}^{\vphantom{\textsf{T}}}\|_{V^\dag}^2 - g_y^\textsf{T}\lambda_n^{\vphantom{\textsf{T}}}\Big] +\\
\textstyle\sum_{n=0}^{N-1}\Big[-\frac{1}{4}\|F_w^\textsf{T}\gamma_n^{\vphantom{\textsf{T}}}- B^\textsf{T}\mu_{n}^{\vphantom{\textsf{T}}} + 2\Phi^\textsf{T}u_n\|_{W_2^{-1}}^2 +\\
\textstyle- g_w^\textsf{T}\gamma_n^{\vphantom{\textsf{T}}} + u_n^\textsf{T}W_1u_n^{\vphantom{\textsf{T}}}\Big]+(F_y\xi_0 - g_y)^\textsf{T}\lambda_0^{\vphantom{\textsf{T}}}-\xi_0^\textsf{T}A^\textsf{T}\mu_0
\end{multline}
and $U^\dag$, $V^\dag$ are the Moore-Penrose pseudoinverse of $U$, $V$.  Our next result shows that the solutions of this duality-based reformulation match the solutions of BiMPC.

\begin{theorem}
Consider the problem $\mathbf{P}_{DB}(\xi_0, \epsilon)$.  We have that $\arg\min \mathbf{P}_B(\xi_0) = \arg\min \mathbf{P}_{DB}(\xi_0, 0)$ and
\begin{equation}
\lim_{\epsilon\rightarrow 0}\mathrm{dist}(\arg\min\mathbf{P}_{DB}(\xi_0, \epsilon), \arg\min \mathbf{P}_B(\xi_0)) = 0,
\end{equation}
where $\mathrm{dist}(\mathcal{S},\mathcal{T}) = \sup_s\inf_t\{\|s - t\|\ |\ s\in\mathcal{S}, t\in\mathcal{T}\}$.
%$\limsup_{\epsilon\rightarrow 0}\arg\min\mathbf{P}_{DB}(\xi_0, \epsilon) \subseteq \arg\min \mathbf{P}_B(\xi_0)$.
\end{theorem}

\begin{proof}
The Langrangian corresponding to LoMPC is $\mathcal{L} = \xi_N^\textsf{T}U\xi_N^{\vphantom{\textsf{T}}} + \textstyle\sum_{n=0}^{N-1}[ \xi_n^\textsf{T}V\xi_n^{\vphantom{\textsf{T}}} + \nu_n^\textsf{T}W\nu_n^{\vphantom{\textsf{T}}} + \mu_n^\textsf{T}(\xi_{n+1}^{\vphantom{\textsf{T}}} - A\xi_n^{\vphantom{\textsf{T}}} - B\nu_n^{\vphantom{\textsf{T}}}) + \lambda_n^\textsf{T}(F_y\xi_n - g_y) + \gamma_n^\textsf{T}(F_w\nu_n - g_w) + \lambda_N^\textsf{T}(F_o\xi_N-g_0)]$, where the $\lambda_n, \mu_n, \gamma_n$ variables have appropriate dimensions.  This Lagrangian has a separable structure, and so we individually consider its minimization in each decision variable. Moreover, each minimization is a convex quadratic program that we solve by setting the gradient in the corresponding decision variable equal to zero.  Then $\arg\inf_{\xi_N} \mathcal{L} \ni\underline{\xi}_N  = -\frac{1}{2}U^\dag(F_o^\textsf{T}\lambda_N^{\vphantom{\textsf{T}}} + \mu_{N-1}^{\vphantom{\textsf{T}}})$ and $\underline{\xi}_N^\textsf{T}U\underline{\xi}_N^{\vphantom{\textsf{T}}} + (F_o^\textsf{T}\lambda_N^{\vphantom{\textsf{T}}} + \mu_{N-1}^{\vphantom{\textsf{T}}})^\textsf{T}\underline{\xi}_N^{\vphantom{\textsf{T}}} = -\frac{1}{4}\|F_o^\textsf{T}\lambda_N^{\vphantom{\textsf{T}}} + \mu_{N-1}^{\vphantom{\textsf{T}}}\|_{U^\dag}^2$.  For $n\in[N-1]$, $ \arg\inf_{\xi_n} \mathcal{L} \ni \underline{\xi}_n = -\frac{1}{2}V^\dag(F_y^\textsf{T}\lambda_n^{\vphantom{\textsf{T}}} + \mu_{n-1}^{\vphantom{\textsf{T}}} - A^\textsf{T}\mu_{n}^{\vphantom{\textsf{T}}})$ and $\underline{\xi}_n^\textsf{T}V\underline{\xi}_n^{\vphantom{\textsf{T}}} + (F_y^\textsf{T}\lambda_n^{\vphantom{\textsf{T}}} + \mu_{n-1}^{\vphantom{\textsf{T}}} - A^\textsf{T}\mu_{n}^{\vphantom{\textsf{T}}})^\textsf{T}\underline{\xi}_n^{\vphantom{\textsf{T}}} = -\frac{1}{4}\|F_y^\textsf{T}\lambda_n^{\vphantom{\textsf{T}}} + \mu_{n-1}^{\vphantom{\textsf{T}}} - A^\textsf{T}\mu_{n}^{\vphantom{\textsf{T}}}\|_{V^\dag}^2$.  For $n\in\langle N\rangle$, $\underline{w}_n = \arg\inf_{w_n} \mathcal{L} = -\frac{1}{2}W_2^{-1}(F_w^\textsf{T}\gamma_n^{\vphantom{\textsf{T}}} - B^\textsf{T}\mu_{n}^{\vphantom{\textsf{T}}} + 2\Phi^\textsf{T}u_n)$ and $\underline{w}_n^\textsf{T}W_2\underline{w}_n^{\vphantom{\textsf{T}}} + (F_w^\textsf{T}\gamma_n^{\vphantom{\textsf{T}}} - B^\textsf{T}\mu_{n}^{\vphantom{\textsf{T}}} + 2\Phi^\textsf{T}u_n)^\textsf{T}\underline{w}_n^{\vphantom{\textsf{T}}} = -\frac{1}{4}\|F_w^\textsf{T}\gamma_n^{\vphantom{\textsf{T}}}- B^\textsf{T}\mu_{n}^{\vphantom{\textsf{T}}} + 2\Phi^\textsf{T}u_n\|_{W_2^{-1}}^2$.  Combining these intermediate calculations shows that the Lagrange dual function is given by the function $\zeta(\xi_0,\mathbf{u},\boldsymbol\mu,\boldsymbol\lambda,\boldsymbol\gamma)$ as defined earlier.  Since the constraints of LoMPC are all linear, strong duality holds \cite{boyd2004} and so $\mathbf{P}_B(\xi_0)$ is equivalent to $\mathbf{P}_{DB}(\xi_0, 0)$.  The final part of the result follows by applying epi-convergence theory, similar to the proofs \cite{aswani2015,aswani2016}.  Specifically, it follows by combining Proposition 7.4.d and Theorem 7.31 of \cite{rockafellar2009variational}.
%\begin{equation}
%\begin{aligned}
%&\textstyle\underline{\xi}_N = \arg\inf_{\xi_N} \mathcal{L} = -\frac{1}{2}U^\dag(F_o^\textsf{T}\lambda_N^{\vphantom{\textsf{T}}} + \mu_{N-1}^{\vphantom{\textsf{T}}})\\
%&\underline{\xi}_N^\textsf{T}U\underline{\xi}_N^{\vphantom{\textsf{T}}} + (F_o^\textsf{T}\lambda_N^{\vphantom{\textsf{T}}} + \mu_{N-1}^{\vphantom{\textsf{T}}})^\textsf{T}\underline{\xi}_N^{\vphantom{\textsf{T}}} = -\frac{1}{4}(F_o^\textsf{T}\lambda_N^{\vphantom{\textsf{T}}} + \mu_{N-1}^{\vphantom{\textsf{T}}})^\textsf{T}U^\dag(F_o^\textsf{T}\lambda_N^{\vphantom{\textsf{T}}} + \mu_{N-1}^{\vphantom{\textsf{T}}})
%\end{aligned}
%\end{equation}
\end{proof}

%The purpose of the $\epsilon$ variable in the reformulation $\mathbf{P}_{DB}(\xi_0, \epsilon)$ is to provide numerical regularization.  When $\epsilon > 0$, the reformulation has certain improved numerical properties \cite{aswani2015,aswani2016}.  And the above result shows that solving $\mathbf{P}_{DB}(\xi_0, \epsilon)$ with an arbitrarily small $\epsilon$ generates a solution that is close to the solution of the original BiMPC problem.
The $\epsilon$ in the reformulation $\mathbf{P}_{DB}(\xi_0, \epsilon)$ provides numerical regularization, and setting $\epsilon > 0$ ensures certain improved numerical properties \cite{aswani2015,aswani2016}.  The above result shows that solving $\mathbf{P}_{DB}(\xi_0, \epsilon)$ with a sufficiently small $\epsilon$ generates a solution close to the solution of the original BiMPC problem.

\subsection{Example: Simulation of Two-Dimensional System}

Consider a situation where the LoMPC is
\begin{equation}
\begin{aligned}
\mathbf{P}_L(\xi_0,\mathbf{u}) = \min_{\boldsymbol\xi,\mathbf{w}}\ & y_1^2 + y_0^2 + w_0^2 \\
 \text{s.t. }& x_1 = 2x_0 + y_0 + u_0 \\
& y_1 = 2y_0 + u_0 +w_0 \\
& w_0 \in [-3,3],\, y_1 \in [-1,1]
\end{aligned}
\end{equation}
Using the synthesis procedure from Sect. \ref{section:scfs} we can compute $\Gamma^\textsf{T} = \frac{1}{2}\cdot\begin{bmatrix}2 & 1\end{bmatrix}$, choose a stabilizing $G = -\frac{1}{2}\cdot\begin{bmatrix}3 & 1\end{bmatrix}$, and compute $H = \frac{1}{12}\begin{bmatrix}\hphantom{-}59  & -10\\ -10 & \hphantom{-}44\end{bmatrix}$. Theorem \ref{thm:synstab} implies that
\begin{equation}
\label{eqn:bimpcsim}
\begin{aligned}
\mathbf{P}_B(\xi_0) = \min_{\boldsymbol\xi,\mathbf{u},\mathbf{w}} \ & \xi_1^\textsf{T}\xi_1^{\vphantom{\textsf{T}}} + \xi_0^\textsf{T}\xi_0^{\vphantom{\textsf{T}}} +  \nu_0^\textsf{T}\nu_0^{\vphantom{\textsf{T}}} \\
\text{s.t. } &(\boldsymbol\xi,\mathbf{w}) \in \arg\min_{\boldsymbol\xi,\mathbf{w}} \mathbf{P}_L(\xi_0,\mathbf{u}) \\
&\xi_1^\textsf{T}H\xi_1^{\vphantom{\textsf{T}}} - \xi_0^\textsf{T}(H-\mathbb{I})\xi_0^{\vphantom{\textsf{T}}} \leq 0\\
& u_0 \in [-2, 2],\, x_1 \in [-1,1]
\end{aligned}
\end{equation}
is stabilizing.  Simulation results where the control action of BiMPC was computed using the duality-based reformulation (\ref{eqn:dbreform}) with regularization of $\epsilon = 0.01$ are shown in Fig. \ref{fig:num_phase}.

\begin{figure}[t]
\centering
\includegraphics[trim={2.35in 4in 2.1in 4.2in},clip,scale=.9]{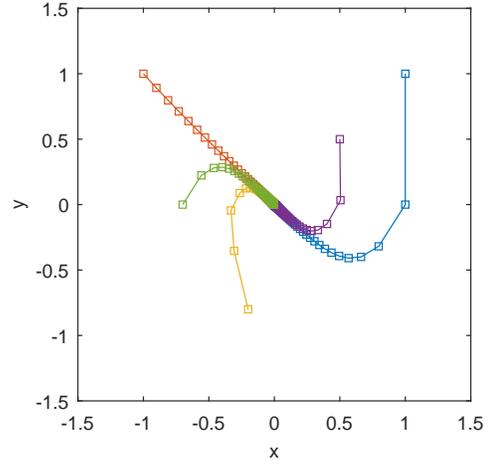}
\caption{A phase plot (with 5 initial conditions) for the overall system with BiMPC in (\ref{eqn:bimpcsim}) solved using our duality-based reformulation (\ref{eqn:dbreform}).}
\label{fig:num_phase}
\end{figure}

\section{Integer-Based Approach to Solving BiMPC}

Another approach to solving bilevel programs is to use mixed-integer programming \cite{aswani2016behavioral,aswani2016b}, and here we describe how to adapt these approaches to solve BiMPC. 

\subsection{Integer-Programming Reformulation of BiMPC}

%Let $\kappa > 0$ be a constant, and define $\boldsymbol\sigma = \{\sigma_0,\ldots,\sigma_{N}\}$ and $\boldsymbol\tau = \{\tau_0,\ldots,\tau_{N-1}\}$.  With this notation, we next present our integer-programming reformulation of BiMPC:
Let $\kappa > 0$ be a constant, $\boldsymbol\sigma = \{\sigma_0,\ldots,\sigma_{N}\}$, and $\boldsymbol\tau = \{\tau_0,\ldots,\tau_{N-1}\}$.  Our integer-programming reformulation is
\begin{align}
\label{eqn:ipreform}
&\mathbf{P}_{IP}(\xi_0) = \\
&\begin{aligned}
\min_{\substack{\boldsymbol\xi,\mathbf{u},\mathbf{w}\\\boldsymbol\mu,\boldsymbol\lambda,\boldsymbol\gamma\\\boldsymbol\sigma,\boldsymbol\tau}}\ & \xi_N^\textsf{T}P\xi_N^{\vphantom{\textsf{T}}} + \textstyle\sum_{n=0}^{N-1} \xi_n^\textsf{T}Q\xi_n^{\vphantom{\textsf{T}}} + \nu_n^\textsf{T}R\nu_n^{\vphantom{\textsf{T}}}\\
\text{s.t. }& 2U\xi_N + \mu_{N-1} + F_o^\textsf{T}\lambda_{N}^{\vphantom{\textsf{T}}} = 0\\
&2V\xi_n + \mu_{n-1} -A^\textsf{T}\mu_n^{\vphantom{\textsf{T}}} + F_y^\textsf{T}\lambda_{n}^{\vphantom{\textsf{T}}} = 0 \text{ for } n\in\langle N\rangle\\
&2W_2w_n + 2\Phi u_n - B^\textsf{T}\mu_n^{\vphantom{\textsf{T}}} + F_w^\textsf{T}\gamma_{n}^{\vphantom{\textsf{T}}} = 0 \text{ for } n\in\langle N\rangle\\
&\lambda_n \leq \kappa \sigma_n,\, F_y^\textsf{T}\xi_n^{\vphantom{\textsf{T}}} - g_y \geq -\kappa(1-\sigma_n) \text{ for } n\in\langle N\rangle\\ 
&\lambda_N \leq \kappa \sigma_N,\, F_o^\textsf{T}\xi_N^{\vphantom{\textsf{T}}} - g_o \geq -\kappa(1-\sigma_N)\\ 
&\gamma_n \leq \kappa \tau_n, \, F_w^\textsf{T}w_n^{\vphantom{\textsf{T}}} - g_w \geq -\kappa(1-\tau_n) \text{ for } n\in\langle N\rangle\\ 
&\xi_{n+1} = A\xi_n + B\nu_n \text{ for } n \in\langle N\rangle\\
&\xi_1^\textsf{T}H\xi_1^{\vphantom{\textsf{T}}} - \xi_0^\textsf{T}(H-\mathbb{I})\xi_0^{\vphantom{\textsf{T}}}\leq 0\\
&F_xx_n\leq g_x,\, F_yy_n\leq g_y \text{ for } n\in\langle N\rangle\\
&F_uu_n\leq h_u,\, F_ww_n\leq h_w \text{ for } n\in\langle N\rangle\\
&F_xx_N\leq g_x,\, F_oy_N\leq g_o\\
&\lambda_n \geq 0,\, \gamma_n \geq 0 \text{ for } n\in\langle N\rangle\\
&\boldsymbol\sigma,\boldsymbol\tau \text{ are binary (0/1) valued}\nonumber
\end{aligned}
\end{align}
where $\lambda_n, \mu_n, \gamma_n, \sigma_n, \tau_n$ have the right size.  This mixed-integer quadratically-constrained quadratic program is solved by standard software \cite{gurobi}, and its solutions match BiMPC.

\begin{theorem}
Consider the problem $\mathbf{P}_{IP}(\xi_0)$.  We have that $\arg\min \mathbf{P}_B(\xi_0) = \arg\min \mathbf{P}_{IP}(\xi_0)$ for sufficiently large $\kappa$.
\end{theorem}

\begin{proof}
The dual (\ref{eqn:langdualfun}) is concave quadratic in $(\boldsymbol\mu,\boldsymbol\lambda,\boldsymbol\gamma$); and $\mathcal{X}$, $\mathcal{Y}$, $\mathcal{U}$ are bounded. So we can choose $\kappa$ to bound the norm of a maximizer of $\zeta$ and of $F_y\xi_n - g_y$, $F_ww_n-g_w$ for $n\in\langle N\rangle$ and $F_o\xi_N-g_0$ for feasible points, since $\mathcal{Y}$, $\mathcal{W}$, $\mathcal{Y}_\Omega$ are bounded.  LoMPC is a convex quadratic program, so KKT equals optimality \cite{boyd2004}.  Replacing LoMPC in $\mathbf{P}_B(\xi_0)$ with KKT where complementarity terms $\lambda(F^\textsf{T}\eta-g) = 0$ are replaced by the equivalent $\lambda \leq \kappa \sigma$ and $F^\textsf{T}\eta-g\geq -\kappa(1-\sigma)$ for $\sigma\in\{0,1\}$ shows $\mathbf{P}_B(\xi_0)$ is equivalent to $\mathbf{P}_{IP}(\xi_0)$.
\end{proof}

\subsection{Case Study: Demand Response for Home Air-Conditioner}

%Electric utilities use demand response (DR) to better match the usage and generation of electricity, and one approach is time-of-day pricing to disincentivize electricity usage during peak demand hours.  But there is a question of how the utility should choose the prices.  Here, we use BiMPC to choose pricing for a home with an air-conditioner controlled by linear MPC.  This scenario is motivated by recent work on using MPC to control HVAC \cite{ma2011,deng2010building,aswani2012,aswani2012energy,aswani2012identifying,he2016}, and is similar to the bilevel approach described in \cite{zugno2013}. 
Electric utilities use demand response (DR) to better match the usage and generation of electricity, and one approach is time-of-day pricing to disincentivize electricity usage during peak demand hours.  Here, we use BiMPC to design electricity pricing for a home with an air-conditioner controlled by linear MPC.  This scenario is motivated by recent work on using MPC to control HVAC \cite{ma2011,deng2010building,aswani2012,aswani2012energy,aswani2012identifying,he2016}, and is similar to the bilevel approach described in \cite{zugno2013}. 

In particular, consider a single home that uses the following (simplified) linear MPC to control an air-conditioner:
\begin{multline}
\mathbf{P}_L(\xi_0,\mathbf{u}) = \\
\begin{aligned}
\min_{\boldsymbol\xi,\mathbf{w}}\ & \textstyle\sum_{n=0}^N (\xi_n - T_d)^2 + \Phi u_nw_n \\
\text{s.t. } & \xi_{n+1} = A\xi_n - B w_n + \beta d_n + q \text{ for } n \in \langle N \rangle \\
& \xi_n\in[20,24],\, w_n\in[0,0.5] \text{ for } n \in \langle N+1 \rangle% \text{ for } k \in \langle N \rangle
\end{aligned}
\end{multline} 
where $\xi_n$ is room temperature ($^\circ C$), $T_d$ is desired room temperature, $\Phi$ quantifies the home owner's trade off between comfort and cost, $u_n$ is electricity price (cents/kWh), $w_n$ is the air-conditioner's duty cycle, $d_n$ is outdoor temperature, and $q$ is heating due to occupancy.  The sampling period is 15 minutes, and the parameter values $A=0.64$, $B = 2.64$, $\beta = 0.10$, $q = 6.98$ are from the HVAC model in \cite{aswani2012}.

If the electric utility would like to reduce electricity consumption during 1PM-5PM, then the problem of choosing time-of-day pricing can be written as the BiMPC given by
\begin{equation}
\begin{aligned}
\mathbf{P}_B(\xi_0) = \min_{\boldsymbol\xi,\mathbf{u},\mathbf{w}}\ & \textstyle100\sum_{n \in \mathcal{J}} w_n + \sum_{n=0}^{N-1}u_n \\
\text{s.t. } & (\boldsymbol\xi,\mathbf{w}) \in \arg\min_{\boldsymbol\xi,\mathbf{w}} \mathbf{P}_L(\xi_0,\mathbf{u}) \\
& u_n \in[5,10] \text{ for } n \in \langle N \rangle
\end{aligned}
\end{equation}
%where $H_p$ are the indices that correspond to 1PM-5PM.  This BiMPC with this can be solved using integer programming using the reformulation technique described in \cite{aswani2016b,aswani2016behavioral}. The approach is to rewrite the constraint $(T_k,u_k) \in \argmin \mathbf{P}_L(T_0,c_k)$ as the KKT conditions \cite{boyd2004}, and then use binary variables to exactly linearize the reformulation.  For $\mathbf{P}_B(T_0)$, the resulting integer program is
where $\mathcal{J}$ are the indices that correspond to 1PM-5PM.  The integer-programming reformulation (\ref{eqn:ipreform}) for this BiMPC was solved with Gurobi \cite{gurobi} and CVX \cite{cvx} in MATLAB R2016b.  Simulation results over one day with weather data from \cite{noaa} are shown in Fig. \ref{fig:dr_sim_fig}, and the chosen price induces the HVAC controller to precool the room to reduce electricity consumption between 1PM-5PM (which was the DR goal of the electric utility). The solution time on a laptop computer with a 2.4GHZ processor and 16GB RAM was on average 2.55s, with a minimum of 0.56s and maximum of 8.95s.

\begin{figure*}[t]
	\centering
	\includegraphics[trim={0.7in 4in 0.4in 4.1in},clip,scale=0.9]{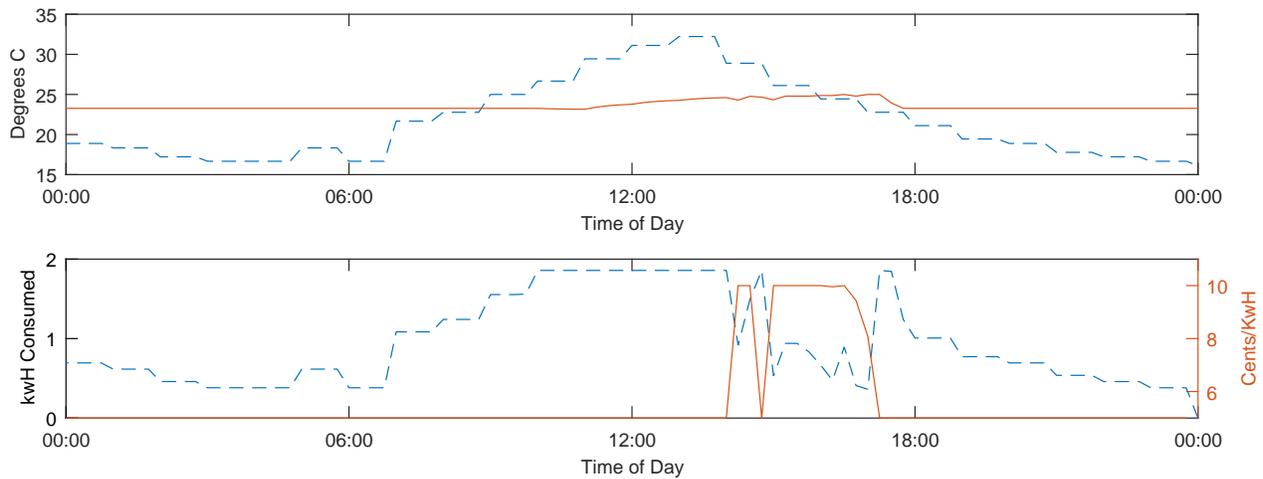}
	\caption{A simulation of the BiMPC controller shows reduced electricity consumption between 1PM-5PM by choosing electricity time-of-day pricing that induces the linear MPC of the HVAC to precool the room in the morning.  The top plot shows room temperature (solid red) and outdoor temperature (dashed blue), and the bottom plot shows electricity price (solid red) and HVAC energy consumption (dashed blue).}
	\label{fig:dr_sim_fig}
\end{figure*}

\section{Conclusion}

In this paper, we defined BiMPC, gave examples that show interconnections in dynamic Stackelberg games can lead to loss/gain of controllability or stability, provided sufficient conditions under an arbitrary finite MPC horizon for stabilizability of BiMPC, and developed an approach to synthesize a stabilizing BiMPC controller.  We derived duality-based and integer-programming-based techniques for numerically solving the optimization problem associated with BiMPC, and demonstrated these reformulations with simulations.

\bibliographystyle{IEEEtran}
\bibliography{IEEEabrv,bilmpc}

% Generated by IEEEtran.bst, version: 1.14 (2015/08/26)
\begin{thebibliography}{10}
\providecommand{\url}[1]{#1}
\csname url@samestyle\endcsname
\providecommand{\newblock}{\relax}
\providecommand{\bibinfo}[2]{#2}
\providecommand{\BIBentrySTDinterwordspacing}{\spaceskip=0pt\relax}
\providecommand{\BIBentryALTinterwordstretchfactor}{4}
\providecommand{\BIBentryALTinterwordspacing}{\spaceskip=\fontdimen2\font plus
\BIBentryALTinterwordstretchfactor\fontdimen3\font minus
  \fontdimen4\font\relax}
\providecommand{\BIBforeignlanguage}[2]{{%
\expandafter\ifx\csname l@#1\endcsname\relax
\typeout{** WARNING: IEEEtran.bst: No hyphenation pattern has been}%
\typeout{** loaded for the language `#1'. Using the pattern for}%
\typeout{** the default language instead.}%
\else
\language=\csname l@#1\endcsname
\fi
#2}}
\providecommand{\BIBdecl}{\relax}
\BIBdecl

\bibitem{camponogara2002}
E.~Camponogara, D.~Jia, B.~H. Krogh, and S.~Talukdar, ``Distributed model
  predictive control,'' \emph{IEEE Control Systems}, vol.~22, no.~1, pp.
  44--52, 2002.

\bibitem{venkat2005}
A.~N. Venkat, J.~B. Rawlings, and S.~J. Wright, ``Stability and optimality of
  distributed model predictive control,'' in \emph{Proc. of IEEE CDC}, 2005,
  pp. 6680--6685.

\bibitem{rawlings2009}
J.~Rawlings and D.~Mayne, \emph{Model Predictive Control: Theory and
  Design}.\hskip 1em plus 0.5em minus 0.4em\relax Nob Hill Pub., 2009.

\bibitem{scattolini2009}
R.~Scattolini, ``Architectures for distributed and hierarchical model
  predictive control--a review,'' \emph{Journal of Process Control}, vol.~19,
  no.~5, pp. 723--731, 2009.

\bibitem{raimondo2009}
D.~M. Raimondo, P.~Hokayem, J.~Lygeros, and M.~Morari, ``An iterative
  decentralized {MPC} algorithm for large-scale nonlinear systems,'' \emph{IFAC
  Proceedings Volumes}, vol.~42, no.~20, pp. 162--167, 2009.

\bibitem{ma2011}
Y.~Ma, G.~Anderson, and F.~Borrelli, ``A distributed predictive control
  approach to building temperature regulation,'' in \emph{Proc. of IEEE ACC},
  2011, pp. 2089--2094.

\bibitem{farina2011}
M.~Farina and R.~Scattolini, ``Distributed non-cooperative {MPC} with
  neighbor-to-neighbor communication,'' \emph{IFAC Proceedings Volumes},
  vol.~44, no.~1, pp. 404--409, 2011.

\bibitem{ferramosca2013}
A.~Ferramosca, D.~Lim{\'o}n, I.~Alvarado, and E.~F. Camacho, ``Cooperative
  distributed {MPC} for tracking,'' \emph{Automatica}, vol.~49, no.~4, pp.
  906--914, 2013.

\bibitem{scattolini2007}
R.~Scattolini and P.~Colaneri, ``Hierarchical model predictive control,'' in
  \emph{Proc. of IEEE CDC}, 2007, pp. 4803--4808.

\bibitem{picasso2010}
B.~Picasso, D.~De~Vito, R.~Scattolini, and P.~Colaneri, ``An {MPC} approach to
  the design of two-layer hierarchical control systems,'' \emph{Automatica},
  vol.~46, no.~5, pp. 823--831, 2010.

\bibitem{vermillion2014}
C.~Vermillion, A.~Menezes, and I.~Kolmanovsky, ``Stable hierarchical model
  predictive control using an inner loop reference model and
  $\lambda$-contractive terminal constraint sets,'' \emph{Automatica}, vol.~50,
  no.~1, pp. 92--99, 2014.

\bibitem{von1952}
H.~von Stackelberg, \emph{The Theory of the Market Economy}.\hskip 1em plus
  0.5em minus 0.4em\relax Oxford University Press, 1952.

\bibitem{basar1979}
T.~Basar and H.~Selbuz, ``Closed-loop {Stackelberg} strategies with
  applications in the optimal control of multilevel systems,'' \emph{IEEE TAC},
  vol.~24, no.~2, pp. 166--179, 1979.

\bibitem{li2002}
M.~Li, J.~Cruz, and M.~A. Simaan, ``An approach to discrete-time incentive
  feedback {Stackelberg} games,'' \emph{IEEE Trans. Syst., Man, Cybern. A,
  Syst.,Humans}, vol.~32, no.~4, pp. 472--481, 2002.

\bibitem{aswani2011}
A.~Aswani and C.~Tomlin, ``Game-theoretic routing of {GPS}-assisted vehicles
  for energy efficiency,'' in \emph{Proc. of ACC}, 2011, pp. 3375--3380.

\bibitem{zhu2011}
M.~Zhu and S.~Mart{\'\i}nez, ``Stackelberg-game analysis of correlated attacks
  in cyber-physical systems,'' in \emph{ACC}, 2011, pp. 4063--4068.

\bibitem{vasudevan2012safe}
R.~Vasudevan, V.~Shia, Y.~Gao, R.~Cervera-Navarro, R.~Bajcsy, and F.~Borrelli,
  ``Safe semi-autonomous control with enhanced driver modeling,'' in
  \emph{ACC}, 2012, pp. 2896--2903.

\bibitem{krichene2014}
W.~Krichene, J.~D. Reilly, S.~Amin, and A.~M. Bayen, ``Stackelberg routing on
  parallel networks with horizontal queues,'' \emph{IEEE TAC}, vol.~59, no.~3,
  pp. 714--727, 2014.

\bibitem{jamaludin2015bilevel}
M.~Z. Jamaludin and C.~L. Swartz, ``A bilevel programming formulation for
  dynamic real-time optimization∗∗ this work is sponsored by the mcmaster
  advanced control consortium (macc) and the ministry of higher education
  (mohe), malaysia,'' \emph{IFAC-PapersOnLine}, vol.~48, no.~8, pp. 906--911,
  2015.

\bibitem{sadigh2016planning}
D.~Sadigh, S.~Sastry, S.~A. Seshia, and A.~D. Dragan, ``Planning for autonomous
  cars that leverages effects on human actions,'' in \emph{Proc. of RSS}, 2016.

\bibitem{colson2007}
B.~Colson, P.~Marcotte, and G.~Savard, ``An overview of bilevel optimization,''
  \emph{Annals of Operations Research}, vol. 153, no.~1, pp. 235--256, 2007.

\bibitem{aswani2015}
A.~Aswani, Z.-J.~M. Shen, and A.~Siddiq, ``Inverse optimization with noisy
  data,'' \emph{arXiv preprint arXiv:1507.03266}, 2015.

\bibitem{aswani2016}
A.~Ouattara and A.~Aswani, ``Duality approach to bilevel programs with a convex
  lower level,'' \emph{arXiv preprint arXiv:1608.03260}, 2016.

\bibitem{aswani2016behavioral}
A.~Aswani, P.~Kaminsky, Y.~Mintz, E.~Flowers, and Y.~Fukuoka, ``Behavioral
  modeling in weight loss interventions,'' \emph{Available at SSRN 2838443},
  2016.

\bibitem{deng2010building}
K.~Deng, P.~Barooah, P.~G. Mehta, and S.~P. Meyn, ``Building thermal model
  reduction via aggregation of states,'' in \emph{Proc. of IEEE ACC}, 2010, pp.
  5118--5123.

\bibitem{aswani2012}
A.~Aswani, N.~Master, J.~Taneja, D.~Culler, and C.~Tomlin, ``Reducing transient
  and steady state electricity consumption in {HVAC} using learning-based
  model-predictive control,'' \emph{Proc. IEEE}, vol. 100, no.~1, pp. 240--253,
  2012.

\bibitem{aswani2012energy}
A.~Aswani, N.~Master, J.~Taneja, A.~Krioukov, D.~Culler, and C.~Tomlin,
  ``Energy-efficient building {HVAC} control using hybrid system {LBMPC},''
  \emph{IFAC Conf. on NMPC}, vol.~45, no.~17, pp. 496--501, 2012.

\bibitem{aswani2012identifying}
A.~Aswani, N.~Master, J.~Taneja, V.~Smith, A.~Krioukov, D.~Culler, and
  C.~Tomlin, ``Identifying models of {HVAC} systems using semiparametric
  regression,'' in \emph{Proc. of IEEE ACC}, 2012, pp. 3675--3680.

\bibitem{he2016}
R.~He and H.~Gonzalez, ``Zoned {HVAC} control via {PDE}-constrained
  optimization,'' in \emph{Proc. of IEEE ACC}, 2016, pp. 587--592.

\bibitem{aswani2016b}
Y.~Mintz, A.~Aswani, P.~Kaminsky, E.~Flowers, and Y.~Fukuoka, ``Behavioral
  analytics for myopic agents,'' \emph{arXiv preprint arXiv:1702.05496}, 2017.

\bibitem{mayne1997}
D.~Q. Mayne and W.~Schroeder, ``Robust time-optimal control of constrained
  linear systems,'' \emph{Automatica}, vol.~33, no.~12, pp. 2103--2118, 1997.

\bibitem{kolmanovsky1998}
I.~Kolmanovsky and E.~G. Gilbert, ``Theory and computation of disturbance
  invariant sets for discrete-time linear systems,'' \emph{Mathematical
  problems in engineering}, vol.~4, no.~4, pp. 317--367, 1998.

\bibitem{blanchini1999}
F.~Blanchini, ``Survey paper: Set invariance in control,'' \emph{Automatica},
  vol.~35, no.~11, pp. 1747--1767, 1999.

\bibitem{rakovic2005}
S.~Rakovic, E.~Kerrigan, K.~Kouramas, and D.~Mayne, ``Invariant approximations
  of the minimal robust positively invariant set,'' \emph{IEEE TAC}, vol.~50,
  no.~3, pp. 406--410, 2005.

\bibitem{mohan2016convex}
S.~Mohan and R.~Vasudevan, ``Convex computation of the reachable set for hybrid
  systems with parametric uncertainty,'' in \emph{Proc. of ACC}, 2016, pp.
  5141--5147.

\bibitem{lu2000quasi}
Y.~Lu and Y.~Arkun, ``Quasi-min-max {MPC} algorithms for {LPV} systems,''
  \emph{Automatica}, vol.~36, no.~4, pp. 527--540, 2000.

\bibitem{mhaskar2006stabilization}
P.~Mhaskar, N.~H. El-Farra, and P.~D. Christofides, ``Stabilization of
  nonlinear systems with state and control constraints using lyapunov-based
  predictive control,'' \emph{Systems \& Control Letters}, vol.~55, no.~8, pp.
  650--659, 2006.

\bibitem{boyd2004}
S.~Boyd and L.~Vandenberghe, \emph{Convex Optimization}.\hskip 1em plus 0.5em
  minus 0.4em\relax Cambridge University Press, 2004.

\bibitem{rockafellar2009variational}
R.~T. Rockafellar and R.~J.-B. Wets, \emph{Variational analysis}.\hskip 1em
  plus 0.5em minus 0.4em\relax Springer, 2009, vol. 317.

\bibitem{gurobi}
\BIBentryALTinterwordspacing
I.~Gurobi~Optimization, ``Gurobi optimizer reference manual,'' 2016. [Online].
  Available: \url{http://www.gurobi.com}
\BIBentrySTDinterwordspacing

\bibitem{zugno2013}
M.~Zugno, J.~M. Morales, P.~Pinson, and H.~Madsen, ``A bilevel model for
  electricity retailers' participation in a demand response market
  environment,'' \emph{Energy Economics}, vol.~36, pp. 182--197, 2013.

\bibitem{cvx}
M.~Grant and S.~Boyd, ``{CVX}: Matlab software for disciplined convex
  programming, version 2.1,'' Mar. 2014.

\bibitem{noaa}
\BIBentryALTinterwordspacing
Unedited local climatological data. National Centers for Environmental
  Information (NCEI). [Online]. Available:
  \url{https://www.ncdc.noaa.gov/ulcd/ULCD?prior=Y}
\BIBentrySTDinterwordspacing

\end{thebibliography}

\end{document}